\documentclass{amsart}
\usepackage[utf8]{inputenc}
\usepackage[margin=01.45 in]{geometry}
\usepackage{dirtytalk}
\usepackage[dvipsnames]{xcolor}
\usepackage{float}
\usepackage[usestackEOL]{stackengine}[2013-10-15] 

\usepackage{amsmath}
\usepackage{amssymb}
\usepackage{amsthm}
\usepackage{tikz-cd}
\usepackage{tikz}
\usetikzlibrary{knots, snakes}
\usetikzlibrary{arrows.meta}
\usepgflibrary[arrows.meta]
\usetikzlibrary{angles, quotes, decorations.markings, decorations.pathmorphing, decorations.pathreplacing, decorations.shapes}
\usepackage{graphicx}
\usepackage{mathtools}
\usepackage{extarrows}
\usepackage{booktabs}
\usepackage{pythonhighlight}
\usepackage{hyperref}
\usepackage[all]{xy}
\usepackage{stmaryrd}
\usepackage{mathabx}
\usepackage{svg}
\usepackage{amsmath,scalerel}
\usepackage{comment}
\DeclareMathOperator*\bigcircop{\scalerel*{\bigcirc}{\bigodot}}

\usepackage{amsmath,amsthm,amssymb,latexsym,amscd}
\usepackage{mathrsfs}

\usepackage{hyperref}
\hypersetup{colorlinks=true,linkcolor=blue,anchorcolor=blue,citecolor=blue}
\numberwithin{equation}{section}
\theoremstyle{plain}

\newtheorem*{lemma*}{Lemma}

\newtheorem{proposition}[equation]{Proposition}
\newtheorem*{proposition*}{Proposition} 

\newtheorem{theorem}[equation]{Theorem}
\newtheorem*{theorem*}{Theorem}

\newtheorem{corollary}[equation]{Corollary}
\newtheorem*{corollary*}{Corollary}

\newtheorem*{conjecture*}{Conjecture}

\definecolor{purp}{RGB}{148,30,238}
\definecolor{pinky}{RGB}{255, 174, 252}

\definecolor{purp1}{RGB}{120,30,238}
\definecolor{pinky1}{RGB}{243, 162, 252}
\definecolor{pinky2}{RGB}{252, 165, 249}
\definecolor{pinky3}{RGB}{247, 174, 255}

\definecolor{purp2}{RGB}{189, 144, 249}
\definecolor{purp3}{RGB}{203, 172, 243}

\theoremstyle{remark}
\newtheorem{claim}[equation]{Claim}

\newtheorem{remark}[equation]{Remark}

\theoremstyle{definition}

\newtheorem{definition}[equation]{Definition}

\DeclareMathOperator{\Span}{span}

\def\R{\mathbb R}

\def\wt{\widetilde}

\stackMath 

\title{A Note on Alternating knots in handlebodies}
\author{Lizzie Buchanan} \email{elizabeth-buchanan@uiowa.edu} \address{University of Iowa, Iowa, USA}
\author{Tanushree Shah} \email{tanushrees@cmi.ac.in}\address{Chennai Mathematical Institute, India}

\date{\today}
\keywords{}

\subjclass[2020]{}

\begin{document}
\begin{abstract}  
 We establish a Kauffman-Murasugi-Thistlethwaite-type theorem for alternating knots in a solid torus. Specifically, we show that any \textit{dotted-reduced} alternating diagram of a knot in a handlebody realizes the minimal crossing number, and that any two such diagrams of the same knot have identical writhe. The proof relies on a generalization of the Jones polynomial to the setting of handlebodies. A stronger version of this result was already proved by Boden, Karimi, and Sikora using a different generalized Jones polynomial; therefore, this text largely expands on one of the main proof tools. 
\end{abstract} 
\maketitle

\section{Introduction}
In the late 19th century, Tait formulated several conjectures concerning alternating knot diagrams, motivated by empirical observations and early attempts to classify knots. One of these, the Reduced Alternating Diagram Conjecture, asserts that any two reduced alternating diagrams of the same knot have the same number of crossings. Equivalently, the minimal crossing number of an alternating knot is achieved by any of its reduced alternating diagrams \cite{Tai}.

This conjecture remained unresolved for more than a hundred years, until it was finally proved independently by Kauffman \cite{Kauffman87}, Murasugi \cite{Murasugi87}, and Thistlethwaite \cite{Thistlethwaite87} in 1987. Their proofs made essential use of the Jones polynomial, showing that its span determines the crossing number for alternating knots. The resolution of Tait's conjecture not only confirmed his geometric intuition but also marked a key moment in the development of knot theory, demonstrating the power of quantum invariants in low-dimensional topology.

Every 3-manifold admits a Heegaard decomposition, dividing it into two handlebodies of $g$-genus. The decomposition represents the manifold as the union of these two handlebodies, glued along the common boundary surface. This decomposition provides a framework to project knots in $M^3$ onto the handlebodies. We have studied reduced alternating knots in genus 1-handle bodies, and a notion of \textit{dotted-reduced} diagrams (see Definition \ref{dotted reduced def}). Specifically we prove for such diagrams that\\

\noindent \textbf{Corollary 
\ref{main result}.}
    \textit{An alternating dotted-reduced diagram of a link in a solid torus is minimal (with respect to crossing number).}\\

We prove this by studying a generalization of the Jones polynomial developed by Bataineh and Hajij in \cite{BH}, for links in the solid torus. This polynomial is very similar to but not the same as the generalized Jones polynomial appearing in \cite{BKS} or Jones-Krushkal polynomial of  \cite{KRUSHKAL_2011} and \cite{BH22}.

We note that Corollary \ref{main result} is already proved in \cite{BKS} as a more general result, but we specialize to handlebodies here, studying the links in this setup so that we can later use these tools and this language to study links in arbitrary 3-manifolds.

We define generalized reduced diagrams (which can be detected from the projection) for knots in handlebodies and use the span of the generalized Jones polynomial for this study. These same ideas extend to studying alternating knots in genus g-handle bodies. .

\subsection*{Acknowledgements} The first author is partially supported by NSF Grant DMS-2038103 at the University of Iowa. The second author receives partial support from the Infosys Fellowship.
\section{Preliminary}
In this section we give a description of a Jones polynomial for links in the solid torus, and we refer the reader to \cite{BH} for more details.

Let \( M = S^1 \times D^2 \) be the solid torus, viewed as a handlebody of genus one. A link \( L \subset M \) can be studied via its projection onto a punctured disk \( D^2 \setminus \{0\} \), where the puncture corresponds to the core \( S^1 \times \{0\} \) of the solid torus.

To define a Jones-type polynomial invariant in this setting, we work within the \emph{Kauffman bracket skein module} of a link diagram $D$ in the solid torus developed by Bataineh and Hajij in \cite{BH}. The bracket $\langle \cdot \rangle$ in the variables $A$ and $t$ is defined by: 
$$\Big \langle 
\begin{tikzpicture}[scale=0.2]
  \draw[thin] (-1,1) -- (-0.2,0.2);
  \draw[thin] (0.2,-0.2) -- (1,-1);
  \draw[thin] (-1,-1) -- (1,1);
\end{tikzpicture} \Big \rangle= A \Big \langle )( \Big \rangle + A^{-1}\Big \langle \rotatebox[origin=c]{90}{)(}  \Big \rangle
$$

$$\Big \langle  D  \cup \bigcircop \Big \rangle = (-A^2 - A^{-2}) \langle  D  \rangle$$

$$\Big \langle  D  \cup \bigodot \Big \rangle = (-A^2 - A^{-2}) \hspace{2pt} t \hspace{2pt} \langle  D  \rangle$$

$$\Big \langle \bigcircop \Big \rangle = 1$$

$$\Big \langle \bigodot \Big \rangle = t$$

Let \( R = \mathbb{Z}[A^{\pm1}, t] \). The \emph{Kauffman bracket skein module} \( \mathcal{S}(M) \) is the \( R \)-module generated by isotopy classes of framed unoriented links in \( M \), modulo the Kauffman bracket relations shown above.

The module \( \mathcal{S}(M) \) has a natural basis indexed by the number of parallel copies of the core curve of the torus. That is, any link can be resolved via the skein relations into a linear combination of powers of a generator \( t \), where \( t \) denotes a longitude of the torus.

The skein module of the solid torus is the free $\mathbb{Z}[A^{\pm 1}]$--module 
generated by powers of $t$, and the above sum gives a well-defined skein element.  
For an oriented diagram $D$ with writhe $w(D)$, we normalize as in the case of $S^{3}$ to obtain the Jones polynomial.

The \emph{Jones polynomial} \( V_L(q, t) \) for a framed link \( L \subset S^1 \times D^2 \) is derived from the normalized Kauffman bracket as follows:
\begin{itemize}
  \item Let \( \langle D \rangle \in \mathcal{S}(S^1 \times D^2) \) be the Kauffman bracket of a diagram $D$ of a link \( L \), expressed in the basis \( \{1, t, t^2, \dots\} \). 

  \item Each basis element \( t^n \) corresponds to \( n \) parallel longitudes (a cable of the core).
 
  \item Apply the writhe correction and change variables to obtain a Jones polynomial for links in the thickened surface, invariant among all diagrams of the link $L$:
  \[
  V_L(q, t) = (-A^3)^{-w(D)} \langle D \rangle \Big|_{A = q^{-1/4}}
  \]
  where \( w(D) \) is the \textit{writhe} of the diagram (the number of positive crossings  \begin{tikzpicture}[every path/.style={thick}, every
node/.style={transform shape, knot crossing, inner sep=5pt}, scale=0.23]
    \node (tl) at (-1, 1) {};
    \node (tr) at (1, 1) {};
    \node (bl) at (-1, -1) {};
    \node (br) at (1, -1) {};
    \node (c) at (0,0) {};

    \draw [-{Stealth[length=4pt,width=4pt]}] (bl) -- (tr);
    \draw (br) -- (c);
    \draw [-{Stealth[length=4pt,width=4pt]}] (c) -- (tl);
    \end{tikzpicture} 
    minus the number of negative crossings
     \begin{tikzpicture}[every path/.style={thick}, every
node/.style={transform shape, knot crossing, inner sep=5pt}, scale=0.23]
    \node (tl) at (-1, 1) {};
    \node (tr) at (1, 1) {};
    \node (bl) at (-1, -1) {};
    \node (br) at (1, -1) {};
    \node (c) at (0,0) {};

    \draw [-{Stealth[length=4pt,width=4pt]}] (br) -- (tl);
    \draw (bl) -- (c);
    \draw [-{Stealth[length=4pt,width=4pt]}] (c) -- (tr);
    \end{tikzpicture}
    ). 
\end{itemize}
When we only consider one link $L$, we will generally write simple $V$ instead of $V_L$. 

Equivalently, let $S$ be a Kauffman state of $D$, i.e.\ a choice of $A$- or $B$-smoothing 
at each crossing.  Write $a(S)$ and $b(S)$ for the number of $A$- and $B$-smoothings, respectively.

\begin{figure}[h]
    \centering
    \begin{tikzpicture}[every path/.style={thick}, every
node/.style={transform shape, knot crossing, inner sep=2.75pt}, scale=0.7]
    \node (tl) at (-1, 1) {};
    \node (tr) at (1, 1) {};
    \node (bl) at (-1, -1) {};
    \node (br) at (1, -1) {};
    \node (c) at (0,0) {};

    \draw (bl) -- (tr);
    \draw (br) -- (c);
    \draw (c) -- (tl);
    
    \begin{scope}[xshift=4cm]
    \node (tl) at (-1, 1) {};
    \node (tr) at (1, 1) {};
    \node (bl) at (-1, -1) {};
    \node (br) at (1, -1) {};

    \draw (bl) .. controls (bl.8 north east) and (tl.8 south east) .. (tl);
    \draw (br) .. controls (br.8 north west) and (tr.8 south west) .. (tr);
    \end{scope}

    \begin{scope}[xshift=8cm]
    \node (tl) at (-1, 1) {};
    \node (tr) at (1, 1) {};
    \node (bl) at (-1, -1) {};
    \node (br) at (1, -1) {};

    \draw (bl) .. controls (bl.8 north east) and (br.8 north west) .. (br);
    \draw (tl) .. controls (tl.8 south east) and (tr.8 south west) .. (tr);
    \end{scope}

    \end{tikzpicture}
    \caption{ A crossing (left), its $A$-smoothing (middle), and $B$-smoothing (right)}
    \label{A-smooth and B-smooth}
\end{figure}
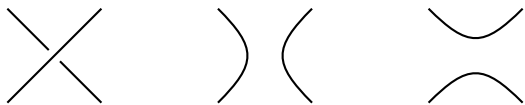
 
Then we express $\langle D \rangle$ as a sum over the contributions from each state $S_\sigma$ as 
\begin{equation}\label{contribution of each state}
   \langle D\rangle \;=\; 
\sum_{\sigma}
A^{\,a(S_\sigma)-b(S_\sigma)}\,
(-A^{2}-A^{-2})^{|S_\sigma|-1}t^{|T_\sigma|}, 
\end{equation}

\noindent where $|S_\sigma|$ is the total number of circles in the state, $|T_\sigma|$ is the number of circles that "go around the dot," $a(S_\sigma)$ is the number of $A$-smoothings in the state, and $b(S_\sigma)$ is the number of $B$-smoothings in the state. \ 

The Jones polynomial in the solid torus is sensitive to winding: links that are isotopic in \( S^3 \) but differ in their winding around the core of \( S^1 \times D^2 \) generally have different Jones polynomials. If a link lies in a 3-ball embedded in the solid torus, its Jones polynomial agrees with the classical Jones polynomial in \( S^3 \).

We can also look at this as $\langle D \rangle = \sum_\sigma \langle D | S_\sigma \rangle$ where the contribution of a single state $S_\sigma$ to $\langle D \rangle$ is

\begin{equation}\label{contrib of state}
    \textit{contribution of state $S_\sigma$} = \langle D | S_\sigma \rangle = A^{a(S_\sigma)-b(S_\sigma)}(-A^{2}-A^{-2})^{|S_\sigma|-1}t^{|T_\sigma|},
\end{equation}

\begin{definition}
    We say that two states $S$ and $S'$ are \textbf{adjacent} if $S'$ is obtained from $S$ by changing the smoothing at exactly one crossing. 
\end{definition}

As it happens with the classical Kauffman bracket, for the bracket polynomial for links in solid torus we have that 

\begin{proposition} \label{potential max and min from SA and SB}
For any link diagram $D$ in solid torus, 
$$\max\deg_A \langle D\rangle \leq \max\deg_A (\text{contribution of the all-$A$ state})= \max \deg_A \langle D | S_A \rangle $$ and $$\min\deg_A \langle D\rangle \geq \min\deg_A (\text{contribution of the all-$B$ state}) = \min\deg_A \langle D | S_B \rangle. $$
\end{proposition}
The proof for this is the same as that of the classical case and hence we omit it. It follows that

\begin{proposition}\label{span bracket leq 4n}
    For any link diagram $D$ in solid torus with $n$ crossings,

    $$\Span \langle D \rangle \leq 4n.$$
\end{proposition}
\begin{proof}
    From Proposition \ref{potential max and min from SA and SB}, we have that 

    \begin{align*}
        \Span V_D &= \max \deg_A \langle D \rangle - \min\deg_A \langle D \rangle \\
        & \leq \max\deg_A \langle D | S_A \rangle - \min\deg_A \langle D | S_B \rangle \\
        & = a(S_A) - b(S_A) + 2(|S_A|-1) - \Big( a(S_B) - b(S_B) - 2(|S_B|-1)\Big)\\
        & = n + 2(|S_A| -1) - \Big(-n - 2(|S_B| - 1) \Big)\\
        & = 2n + 2(|S_A| + |S_B|-2)\\
        & \leq 2n + 2n\\
        & = 4n
    \end{align*}

\end{proof}
\begin{remark}
This is exactly the same proof as in the classical case, included because we will look more closely at some of the inequalities later on.
\end{remark}

\section{Dotted reduced diagrams}

We now introduce the notions of reduced and dotted reduced diagrams for knots in a solid torus.

\subsection{Reduced and dotted-reduced diagrams}

\begin{definition} \label{def of reducible crossing}
    Let $D$ be a link diagram in the solid torus. A crossing $C$ is \textit{nugatory} if we can draw a simple closed curve in the solid torus which separates the solid torus and which intersects $D$ only in the double point of that crossing $c$. 
\end{definition}

This is analogous to the definition of \textit{nugatory} crossings in knots in $S^3$ in classical knot theory. In that setting, a diagram is called \textit{reduced} if it does not have any nugatory crossings. The idea is that any nugatory crossing can be \say{removed} by twisting one half of the diagram. In the setting of a solid torus, not all nugatory crossings can actually be undone that way, as the structure of the solid torus may prevent such twists.

\begin{definition}\label{def of dotted-irreducible crossing}
    Let $D$ be a link diagram in a solid torus, represented as a dotted diagram, with a crossing $c$ that is nugatory as in Definition \ref{def of reducible crossing}. If we can draw a simple closed curve in the solid torus which intersects $D$ only in the double point of crossing $c$, but cannot draw one such that this simple closed curve is contractible in the solid torus, then we say the crossing $c$ is \textit{dotted-irreducible}. Otherwise, we say the crossing is \textit{dotted-reducible} (see Figure \ref{fig:removing dotted-reducible}).
\end{definition}

What we call \textit{dotted-irreducible} crossings for diagrams in a solid torus are comparable to  Boden, Karimi, and Sikora's \textit{essential} crossings for diagrams in a thickened surface in \cite{BKS}.

\begin{definition}\label{dotted reduced def}
    A link diagram in a solid torus is \textit{dotted-reduced} if it has no nugatory crossings, or if its only nugatory crossings are dotted-irreducible crossings. 
\end{definition}

\begin{figure}
    \centering
    \includegraphics[width=4in]{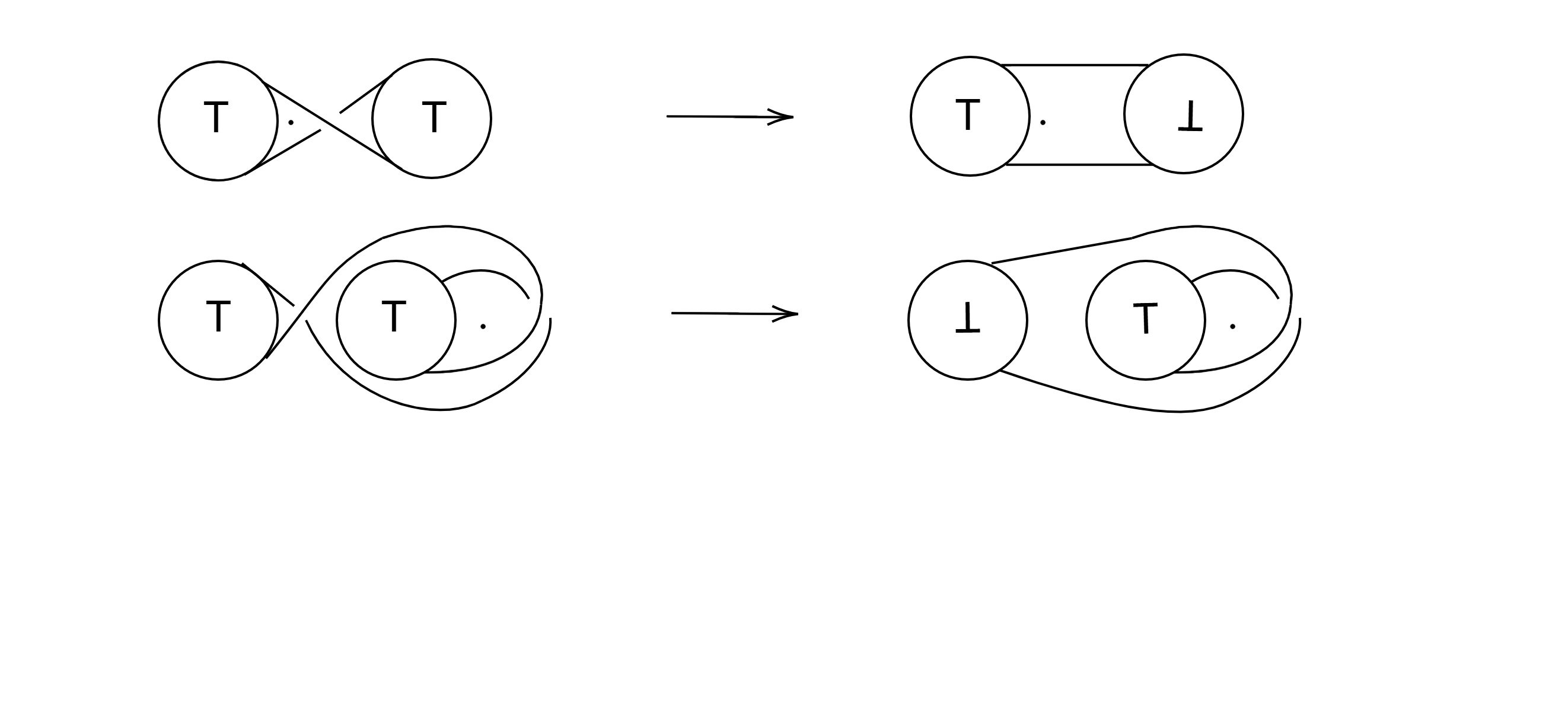}
    \vspace{-2cm}
    \caption{Removing dotted-reducible crossings}
    \label{fig:removing dotted-reducible}
\end{figure}

In the following section, we prove:

\begin{theorem} \label{span bracket and crossing number}
    Let $D$ be an alternating dotted-reduced link diagram in a solid torus with $n$ crossings. Then 
    $$\Span \langle D \rangle = 4n.$$
\end{theorem}

We then relate this Kauffman bracket polynomial for links in a solid torus to the corresponding Jones polynomial for links in a solid torus and obtain: 

\begin{corollary}
\label{span V = crossing number}
    Let $D$ be an alternating dotted-reduced link diagram in a solid torus with $n$ crossings. Then 
    $$\Span V = n(D).$$  
\end{corollary}

\begin{proof}
    This follows from the construction of $V$ from $\langle D \rangle$ and Theorem \ref{span bracket and crossing number}.
\end{proof}

\begin{corollary}\label{main result}
    An alternating dotted-reduced diagram of a link in a solid torus is (minimal with respect to crossing number). 
\end{corollary}

\begin{proof}
    Analogous to proof for classical case: for an alternating dotted-reduced diagram $D$, we have by Corollary \ref{span V = crossing number} that $$\Span V \leq n(L) \leq n(D) = \Span V,$$ so $D$ must be a minimal diagram. 
\end{proof}

\section{Proof of main theorem}

Now, we begin to prove Theorem \ref{span bracket and crossing number}. 



In the proof of Proposition \ref{span bracket leq 4n}, we obtain 
$$\Span \langle D \rangle \leq \max\deg_A\langle D |S_A \rangle - \min\deg_A \langle D |S_B \rangle = 2n + 2(|S_A| + |S_B| -2) \leq 4n.$$

From classical knot theory, we know that if $D$ is an alternating diagram, then $|S_A|+|S_B| = n + 2$ and so that last inequality is an equality \cite{Kauffman87}. 
This holds regardless of whether or not the diagram is reduced (in a classical sense) so it still holds for us in the setting of dotted-reduced. 

To show that in an alternating dotted-reduced diagram we have $\Span \langle D \rangle = 4n$, it remains to show that the first inequality must also be an equality. 
\begin{figure}
    \centering
   \includegraphics[height=7cm]{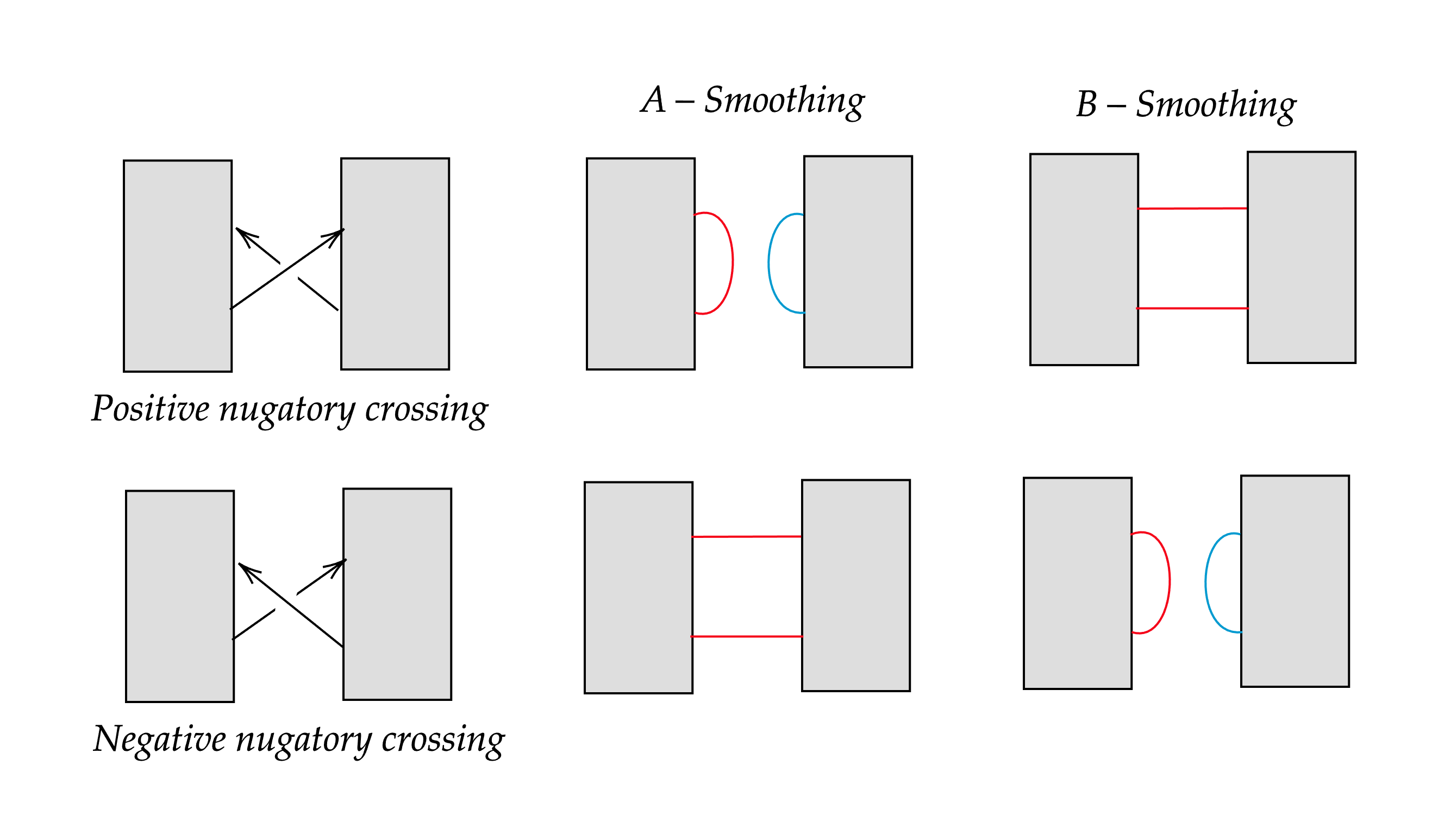}
    \caption{Smoothing crossings}
    \label{SC}
\end{figure}
\begin{theorem} \label{realizes potential max and min}
    Let $D$ be an alternating dotted-reduced link diagram in solid torus. Then  $$\max\deg_A \langle D\rangle = \max\deg_A (\text{contribution of the all-$A$ state})= \max\deg_A \langle D |S_A \rangle$$ and $$\min\deg_A \langle D\rangle = \min\deg_A (\text{contribution of the all-$B$ state}) = \min\deg_A \langle D |S_B \rangle. $$
\end{theorem}

\begin{proof}



By symmetry, it suffices to prove the first statement. We claim that the contribution of the all-$A$ state to the bracket polynomial is not cancelled out by the contribution of any other state. 

    Suppose for contradiction that it does cancel out. Then there is some other state $S_\sigma$ with $\max\deg_A\langle D | S_\sigma \rangle = \max\deg_A \langle D | S_A \rangle$, and so 

    $$\underset{=\max\deg_A\langle D | S_\sigma \rangle}{\underbrace{a(S_\sigma)-b(S_\sigma) + 2(|S_\sigma| - 1)}}  = \underset{=\max\deg_A \langle D | S_A \rangle}{\underbrace{n + 2(|S_A| - 1)}}.$$

    Since $a(S_\sigma)-b(S_\sigma) = n-2b(S_\sigma)$, this is equivalent to $$|S_\sigma| = |S_A| + b(S_\sigma).$$

Consider $\wt{D}$, the image of $D$ under the obvious lift from $A \times I$ to $\R^2 \times I$.

    Since we want to find the total number of $A$-circles, and the number of $A$-circles in $D$ is the same as the number of $A$-circles in $\wt{D}$, we can consider $\wt{D}$ and leverage what we know about classical diagrams. 
    In a reduced alternating classical diagram $\wt{D}$, any state adjacent to $S_A$ will have one fewer circle than $S_A$. However, if the alternating diagram has nugatory crossings, there may be states adjacent to $S_A$ that have more circles.

    \begin{claim}
    Let $\wt{D}$ be a classical (not necessarily reduced) alternating link diagram. 
    State $S_\sigma$ satisfies $|S_\sigma| = |S_A| + b(S_\sigma)$ if and only if $S_\sigma$ is obtained from $S_A$ by performing $B$-smoothings at negative nugatory crossings. 
    \end{claim}

    \begin{proof}
    Let $S_\sigma$ be a state satisfying $|S_\sigma| = |S_A| + b(S_\sigma)$. Then $S_\sigma$ is obtained from $S_A$ by changing the smoothings on a sequence of $b(S_\sigma)$ crossings such that at every step in the sequence, the total number of circles in the resulting state increases by one. 

    Starting with $S_A$ and changing the smoothing of a non-nugatory crossing will decrease the number of circles. So $S_\sigma$ must be obtained by changing the smoothing at a nugatory crossing.

    In Figure \ref{A} we see that a state $S_\sigma$ adjacent to $S_A$ will have $|S_\sigma| = |S_A| + 1$ if and only if $S_\sigma$ is obtained from $S_A$ by performing a $B$-smoothing on a negative nugatory crossing. Similarly, any state $S_\sigma$ will have $|S_\sigma| = |S_A| + b(S_\sigma)$ if and only if $S_\sigma$ is obtained from $S_A$ by performing $B$-smoothings on $b(S_\sigma)$ negative nugatory crossings.  

    \end{proof}

   \begin{claim}
    Let ${D}$ be a dotted-reduced alternating link diagram in a solid torus.
    A state $S_\sigma$ satisfies $|S_\sigma| = |S_A| + b(S_\sigma)$ if and only if $S_\sigma$ is obtained from $S_A$ by performing $B$-smoothings on negative dotted-irreducible crossings. 
   \end{claim}

   \begin{proof}
     Any nugatory crossing in $\wt{D}$ corresponds to a dotted-reducible or dotted-irreducible crossing in $D$. Since our $D$ is dotted-reduced, any nugatory crossing in $\wt{D}$ corresponds to a dotted-irreducible crossing in $D$. So if $S_\sigma$ has the same max $A$-degree as $S_A$, then $S_\sigma$ has $B$-smoothings at some number of negative dotted-irreducible crossings, and has $A$-smoothings at all other crossings. 
   \end{proof}

\begin{figure}
    \centering
   \includegraphics[height=13cm]{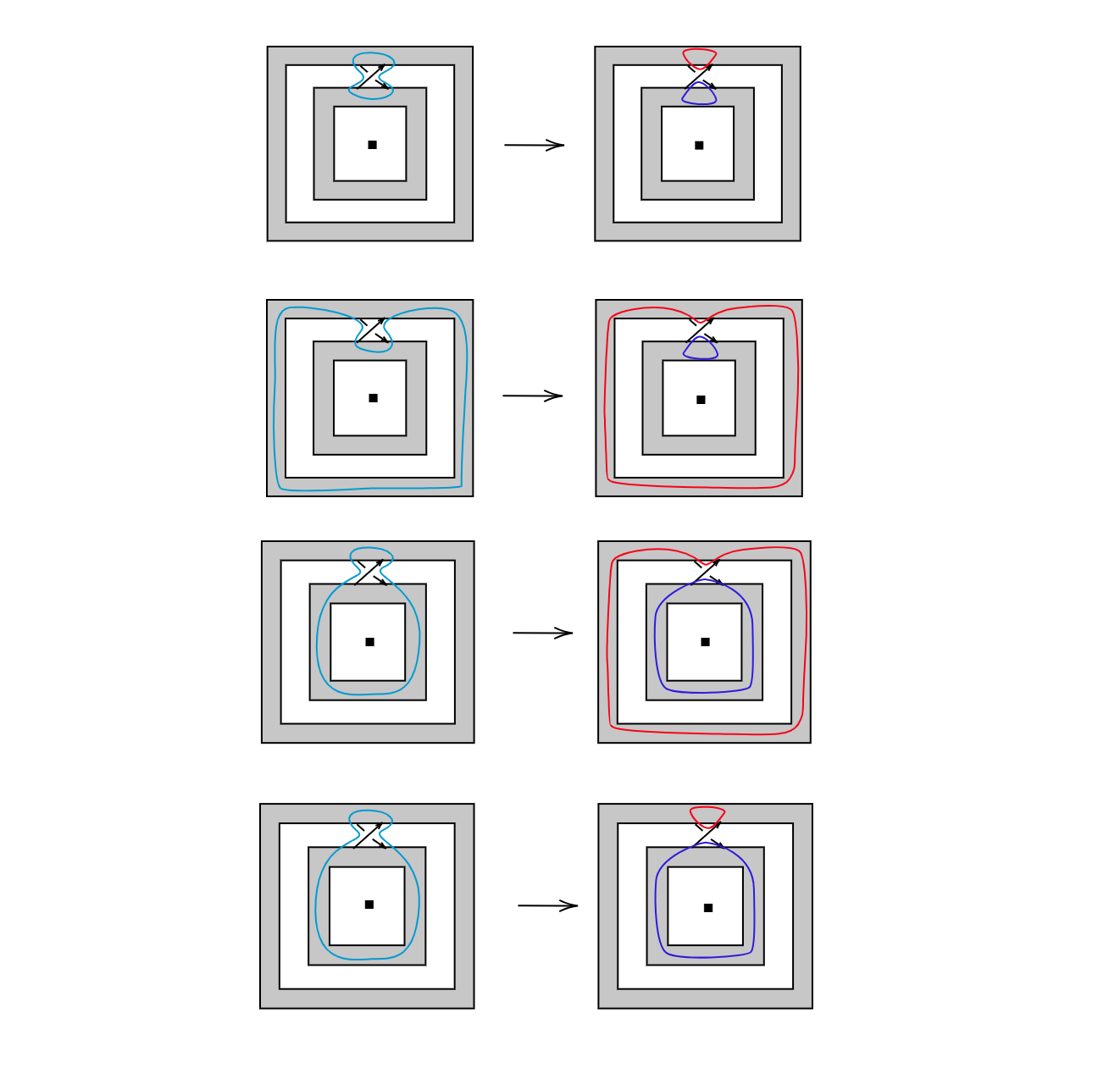}
    \caption{In each of the four rows, we have on the left a state $S$ in which our negative dotted-irreducible crossing is given an $A$-smoothing, and on the right the adjacent state $S'$ which is identical to $S$ except at that one crossing. The four rows show each of the four possibilities for how a state circle involving this crossing could interact with the rest of the diagram.}
    \label{A}
\end{figure}
    \begin{claim}
        If $S_\sigma$ is obtained from $S_A$ by performing $B$-smoothings at negative dotted-irreducible crossings, then $|T_\sigma| > |T_A|$. (As in \ref{contribution of each state} and \ref{contrib of state}, $|T_\sigma|$ is the number of circles that \say{go around the dot} in state $S_\sigma$.)
    \end{claim}

    \begin{proof}

    Consider two states $S_\sigma'$ and $S_\sigma$ that differ only at some negative dotted-irreducible crossing $x$. If $S_\sigma'$ is the state with $x$ $A$-smoothed and $S_\sigma$ is the state with $x$ $B$-smoothed, then either $|T_\sigma| = |T_\sigma'|$ or $|T_\sigma| = |T_\sigma'| + 2$. This is shown in Figure \ref{A-smooth and B-smooth}. So if $S_\sigma$ is obtained from $S_A$ by successively changing the smoothing at negative dotted-irreducible crossings (changing from $A$-smoothings to $B$-smoothings), then $|T_\sigma| \geq |T_A|$.  

    Now consider what happens when we start with the all-$A$ state and change the smoothing at just one negative dotted-irreducible crossing. Call this state $S_\sigma'$. Since the diagram is dotted-reduced and alternating, the $A$-state must look like the example shown in Figure \ref{A}: the circle in the $A$-state that meets itself at this negative dotted-irreducible crossing does not go around the dot. And when we change the smoothing at this crossing to be a $B$-smoothing, we obtain two circles that go around the dot, as in the last row of Figure \ref{A}.  Then the number of circles that go around the dot in $S_\sigma'$ is $|T_\sigma'| = |T_A| + 2$.  

    So if $S_\sigma$ is obtained from $S_A$ by changing some number of negative dotted-irreducible crossings, then $$|T_\sigma| \geq |T_\sigma'| = |T_A| + 2 > |T_A|.$$

    \end{proof}

    Thus any state $S_\sigma$ which has the same $\max\deg_A$ as $S_A$ will not have the same value of $|T|$, so the contributions of $S_\sigma$ and $S_A$ cannot cancel out in the Kauffman bracket for links in the solid torus. Therefore, the contribution of the all-$A$ state to the Jones polynomial does not cancel out, and we have $\max\deg_A \langle D \rangle = \max\deg_A \langle D | S_A \rangle$. 

     This completes the proof of Theorem \ref{realizes potential max and min}.

    \end{proof}

\begin{remark}[Alternating knots in genus $g$-handlebodies]

The Jones polynomial for alternating knots in genus $g$-handlebodies can be obtained analogously to the computation in \cite{BH}. In particular, the same polynomial expression as derived therein applies in our setting, and the proof carries over verbatim. Since no new ideas are required, we omit the details.
\end{remark}

\bibliographystyle{plain}
\bibliography{references.bib}
\vspace{10pt}

\end{document}